\documentclass[10pt,a4paper]{article}
\usepackage[utf8]{inputenc}
\usepackage{amsmath}
\usepackage{amsthm}
\usepackage{amsfonts}
\usepackage{amssymb}
\usepackage{graphicx}
\usepackage{hyperref}
\usepackage{url}
\usepackage{algorithm}
\usepackage{algpseudocode}
\usepackage{pdflscape}
\usepackage{tikz}
\usepackage{longtable}
\usepackage{array}
\usepackage{booktabs}
\usetikzlibrary{matrix}
\usepackage[backend=biber,style=numeric]{biblatex}
\usepackage[left=2cm, right=2cm, top=2cm, bottom=2cm]{geometry}

\title{An attack on Zarankiewicz's problem through SAT solving}
\author{Jeremy Tan\\National University of Singapore}

\newtheorem{theorem}{Theorem}[section]

\newtheorem{corollary}[theorem]{Corollary}
\newtheorem*{argA}{Argument A}
\newtheorem*{argB}{Argument B}
\newtheorem*{argD}{Argument D}
\newtheorem*{argI}{Argument I}

\theoremstyle{definition}
\newtheorem{definition}{Definition}

\begin{document}
	
	\maketitle
	
	\begin{abstract}
		The Zarankiewicz function gives, for a chosen matrix and minor size, the maximum number of ones in a binary matrix not containing an all-one minor. Tables of this function for small arguments have been compiled, but errors are known in them. We both correct the errors and extend these tables in the case of square minors by expressing the problem of finding the value at a specific point as a series of Boolean satisfiability problems, exploiting permutation symmetries for a significant reduction in the work needed.
		
		Certain results related to the graph packing formulation of the problem are used which give exact values at the edges of the function tables. Values in published tables lying deeper in the interior are \textit{not} used, providing independent verification of the correct values in published tables which where almost entirely computed by hand. When the ambient matrix is also square we also give all non-isomorphic examples of matrices attaining the maximum, up to the aforementioned symmetries; it is found that most maximal matrices have some form of symmetry.
	\end{abstract}
	
	\section{Introduction}
	
	The Erdős--Stone theorem \cite{erdosstone} gives an asymptotically tight upper bound for the size of a $H$-free graph of a given order, where $H$ is an arbitrary \textit{non-bipartite} graph. Little is known in the case of bipartite $H$, and to that end Zarankiewicz \cite{origin} posed the following problem in 1951 (translated from the original French):
	
	\begin{quotation}
		Let $R_n$ where $n>3$ be an $n\times n$ square lattice. Find the smallest natural number $k_2(n)$ for which every subset of $R_n$ of size $k_2(n)$ contains 4 points that are all the intersections of 2 rows and 2 columns. More generally, find the smallest natural number $k_j(n)$ for which every subset of $R_n$ of size $k_j(n)$ contains $j^2$ points that are all the intersections of $j$ rows and $j$ columns.
	\end{quotation}
	
	In 1969 Guy \cite{guy} compiled tables of the natural generalisation of $k_j(n)$ where the ambient lattice and the selected sublattice need not be square, but the sublattice cannot be transposed. There is at least one error in his (hand-computed) tables, however, as discovered by Héger \cite{heger}. Merely computing values of $k_j(n)$ also does not provide a complete list of all sublattice-free maximal point sets, which may themselves have many symmetries as the Turán graphs do in their role as extremal $K_n$-free graphs and which may give insights as to the size and structure of further maximal examples.
	
	This paper gives the results of a Boolean satisfiability (SAT)-based approach to Zarankiewicz's problem, motivated by its recent successes in solving very hard combinatorial problems like the fifth Schur number \cite{schur5} and Keller's conjecture in seven dimensions \cite{keller}. Even though much less computational effort was spent here -- all SAT solving was done on a single laptop computer -- already for modestly sized cases the solution is not as trivial as a straight conversion to conjunctive normal form (CNF). The results presented here nevertheless represent a significant extension, both in the range of known values for the non-square generalisation of $k_j(n)$ and (in selected cases) a listing of all maximal examples.
	
	\subsection{Definitions and scope}
	
	\begin{definition}
		The Zarankiewicz function $z_{a,b}(m,n)$ is the maximum number of ones in an $m\times n$ $(0,1)$-matrix with no all-one $a\times b$ minor (such matrices are called \textit{admissible}). Indices are omitted when $a=b$ and when $m=n$, and a matrix achieving the maximum number of ones for a given set of parameters is called \textit{maximal}.
	\end{definition}
	
	Every $(0,1)$-matrix can be interpreted as the biadjacency matrix of a bipartite graph, so $z_a(n)$ is also the maximum size of a $K_{a,a}$-free bipartite graph whose bipartitions have $n$ vertices each. Certain expressions are made simpler with $z_a(n)$ instead of Zarankiewicz's \cite{origin} and Guy's \cite{guy} $k_a(n)=z_a(n)+1$, so the $z$-function will be used in the sequel.
	
	In this paper only the cases $a=b=2,3,4$ will be considered, while maximal matrices will always be discussed up to isomorphism of the equivalent bipartite graphs. The full set of maximal matrices will only be computed when in addition $m=n$; only the value of $z$ is of interest otherwise, with \textit{one} maximal matrix serving as a lower bound complemented by an upper bound proof that adding another one always leads to an all-one $a\times b$ minor.
	
	\section{Exact values, bounds and arguments}\label{sec:eba}
	
	A handful of arguments are listed in Guy \cite{guy} as useful in finding specific values of $z_{a,b}(m,n)$. The three most relevant to this paper are:
	\begin{argA}
		Any admissible matrix with column sums $c_i$, $1\le i\le n$, must satisfy $\sum_i\binom{c_i}a\le(b-1)\binom ma$. Otherwise, by the pigeonhole principle -- where pigeons are $a$-subsets of ones in each column, each such subset potentially part of an all-one $a\times b$ minor, and holes are all $a$-subsets of the matrix's rows -- there is a hole with at least $b$ pigeons, forming an all-one $a\times b$ minor.
	\end{argA}
	\begin{argB}
		For non-negative integers $m,n,k$ with $m-n>1$ and $k\ge2$, $\binom{m-1}k+\binom{n+1}k<\binom mk+\binom nk$. Hence the binomial sum over columns $\sum_i\binom{c_i}a$ in argument A is minimised by distributing ones so that no two column sums differ by more than 1; if the binomial sum is then equal to $(b-1)\binom ma$ and the matrix can still be made to have no all-one $a\times b$ minor, that matrix must be maximal.
	\end{argB}
	\begin{corollary}[Čulik \cite{culik}]
		If $1\le a\le m$ and $n\ge(b-1)\binom ma$, $z_{a,b}(m,n)=(a-1)n+(b-1)\binom ma$.
	\end{corollary}
	\begin{argD}
		Take any row of any admissible matrix. If this row's ones lie in columns with sums $c_1,\dots,c_r$, the inequality $\sum_{i=1}^r\binom{c_i-1}{a-1}\le(b-1)\binom{m-1}{a-1}$ must hold, for otherwise (by argument A) there is an all-one $(a-1)\times b$ minor extendable to an all-one $a\times b$ minor through the ones in the chosen row.
	\end{argD}
	
	The above arguments all have a transposed form obtained by replacing ``columns'' with ``rows'' and vice versa. The following inclusion argument is also clear.
	\begin{argI}
		For $m'\le m$ and $n'\le n$, every $m'\times n'$ minor of every witness to $z_{a,b}(m,n)$ is admissible and thus has at most $z_{a,b}(m',n')$ ones.
	\end{argI}
	
	A useful explicit upper bound, with equality in a wider range of cases than that provided by Čulik's theorem \cite{culik}, is given by the following.
	\begin{theorem}[Roman \cite{roman}]\label{thm:packlimit}
		For all integers $p\ge a-1$
		\begin{equation*}
			z_{a,b}(m,n)\le\left\lfloor\frac{b-1}{\binom p{a-1}}\binom ma+\frac{(p+1)(a-1)}an\right\rfloor
		\end{equation*}
		and equality holds with $p=a$ or $p=a-1$ when $(b-1)\binom ma-aT_{a,b}(m)\le n$, where $T_{a,b}(m)$ is the largest size of a collection $C$ of not necessarily distinct $a+1$-subsets of a set $S$ with $m$ elements such that every $a$-subset of $S$ is covered by at most $b-1$ sets of $C$. The lower bound, which is approximately $\frac{b-1}{a+1}\binom ma$, may be reduced by $a-1$ if the covering is not perfect, i.e.\ $T_{a,b}(m)<\frac{b-1}{a+1}\binom ma$.
	\end{theorem}
	
	This bound appears tighter or at least as tight as other general bounds in the literature, such as the one developed by Collins \cite{collins}, so it is the bound given in the tables in section \ref{sec:tables}.
	
	\subsection{\texorpdfstring{$T_{a,b}(m)$}{T(a,b)(m)}}
	
	The exact value of $T_{a,b}(m)$ for given $a$ and $b$ is an interesting problem in its own right, albeit one that decreases in importance for Zarankiewicz's problem as $a$ and $b$ increase since the $T$-function grows as $O(m^a)$. Guy showed in an earlier paper \cite{guy2} that
	\begin{gather*}
		T_{2,2}(m)=\left\lfloor\frac m3\left\lfloor\frac{m-1}2\right\rfloor\right\rfloor-[m\equiv5\bmod6]\\
		T_{2,b}(m)=\begin{cases}
	    \left\lfloor\frac{b-2}3\binom m2\right\rfloor+T_{2,2}(m)&2\mid m\land2\mid b\\
	    \left\lfloor\frac{b-1}3\binom m2\right\rfloor-[m\equiv5\bmod6\lor b\equiv5\bmod6]&\text{otherwise}
		\end{cases}
	\end{gather*}
	where $[P]$ is the Iverson bracket, evaluating to 1 if the predicate $P$ is true and 0 if $P$ is false, while Bao and Ji \cite{baoji} proved that
	\begin{equation*}
		T_{3,2}(m)=\left\lfloor\frac m4\left\lfloor\frac{m-1}3\left\lfloor\frac{m-2}2\right\rfloor\right\rfloor-[m\equiv0\bmod6]\right\rfloor
	\end{equation*}

    We found no corresponding results in the literature for $T_{3,3}(m)$ and $T_{4,4}(m)$, so to determine their values for small $m$ we used Gurobi (\url{https://gurobi.com}); these values are listed in table \ref{table:coverings}. For brevity only ``base $a+1$-subsets'' and a list of permutations are given for each case -- the actual covering is generated by the action of the permutations on each base subset individually, as done in lemma 2.1 of \cite{baoji}. $T_{3,3}(m)$ cases where $m\equiv2,4\bmod6$ are omitted because then a perfect $T_{3,2}(m)$ covering exists \cite{hanani}, so a perfect $T_{3,3}(m)$ covering can be obtained by duplication.
	
	\newcommand{\wrapp}[1]{\parbox[b]{.35\linewidth}{\vspace{1mm}#1}}
	\begin{table}[ht]
	    \caption{Optimal coverings for $T_{3,3}(m)$ and $T_{4,4}(m)$}\label{table:coverings}
	    \centering
	    \begin{tabular}{lll}
	        Value&Permutations&Base $a+1$-subsets\\
	        \midrule
	        $T_{3,3}(5)=5$&\texttt{(01234)}&\texttt{0123}\\
	        \midrule
	        $T_{3,3}(6)=9$&\texttt{(012345)}&\texttt{0134 0123}\\
	        \midrule
	        $T_{3,3}(7)=15$&\texttt{(01234)(5)(6)}&\texttt{0135 0136 0156}\\
	        \midrule
	        $T_{3,3}(9)=40$&\texttt{(01234567)(8)}&\texttt{0124 0125 0135 0238 0238}\\
	        \midrule
	        $T_{3,3}(11)=80$&\wrapp{\texttt{(01234)(56789)(A), (14)(23)(69)(78)}}&\wrapp{\texttt{0138 013A 018A 0289 0578 0578 0125 0159 0268 056A 068A}}\\
	        \midrule
	        $T_{3,3}(12)=108$&\wrapp{\texttt{(0123456789AB), (1B)(2A)(39)(48)(57)}}&\wrapp{\texttt{0167 0268 0123 0145 0149 0158 0246 0257 0136}}\\
	        \midrule
	        $T_{3,3}(13)=143$&\wrapp{\texttt{(0123456789ABC), (1C)(2B)(3A)(49)(58)(67)}}&\wrapp{\texttt{0159 0167 0269 0124 0139 0146 0258}}\\
	        \midrule
	        $T_{3,3}(15)=225$&\wrapp{\texttt{(0123456789ABCDE), (1E)(2D)(3C)(4B)(5A)(69)(78)}}&\wrapp{\texttt{0123 0145 014C 016A 0178 0257 026B 013B 0169 0248 0258}}\\
	        \midrule
	        $T_{3,3}(17)=340$&\wrapp{\texttt{(0123456789ABCDEFG), (1G)(2F)(3E)(4D)(5C)(6B)(7A)(89)}}&\wrapp{\texttt{013F 014E 0156 018A 0246 0279 027C 037D 0128 013C 014B 0159 025D 036B}}\\
	        \midrule
	        $T_{3,3}(18)=405$&\wrapp{\texttt{(0123456789ABCDEFGH), (1H)(2G)(3F)(4E)(5D)(6C)(7B)(8A)}}&\wrapp{\texttt{029B 039C 049D 0123 014F 0156 0167 0189 025F 028A 0138 014A 015C 0248 025D 026B 036A}}\\
	        \midrule
	        $T_{4,4}(6)=7$&\texttt{(01234)(5)}&\texttt{01234 01234 01235}\\
	        \midrule
	        $T_{4,4}(7)=21$&\texttt{(0123456), (013)(254)}&\texttt{01234}\\
	        \midrule
	        $T_{4,4}(8)=36$&\texttt{(012)(345)(67)}&\wrapp{\texttt{01236 01345 01346 01356 01456 03467 03567 04567}}\\
	        \midrule
	        $T_{4,4}(9)=69$&\texttt{(012345), (67)(8)}&\wrapp{\texttt{01246 01248 01267 01268 01346 01367 01468 02678 03678}}\\
	    \end{tabular}
	\end{table}

    \begin{theorem}\label{thm:tbound}
        \begin{equation*}
	        T_{a,b}(m)\le\left\lfloor\frac m{a+1}\left\lfloor\frac{b-1}a\binom{m-1}{a-1}\right\rfloor\right\rfloor
	    \end{equation*}
    \end{theorem}
    \begin{proof}
        For every $a+1$-subset $E$ in $C$ and any element $v\in E$, exactly $a$ of the $(b-1)\binom{m-1}{a-1}$ available $a$-subsets containing $v$ are covered by $E$, so at most $\left\lfloor\frac{b-1}a\binom{m-1}{a-1}\right\rfloor$ $a+1$-subsets of $C$ can contain $v$. Since $E$ is arbitrary and the number of $E$-$v$ incidences is always a multiple of $a+1$, the claimed upper bound follows.
    \end{proof}
    
    Theorem \ref{thm:tbound} proves the optimality of the $T_{3,3}(m)$ coverings listed in table \ref{table:coverings} except $T_{3,3}(7)$ and $T_{3,3}(11)$, as well as the $T_{4,4}(7)$ covering, since the upper bound is attained in these cases.
	
	\section{Method}\label{sec:method}
	
	Beyond the range of arguments for which theorem \ref{thm:packlimit} gives a proven exact value for the $z$-function, the SAT-based approach calls for encoding an instance of the problem with $a,b,m,n$ and a guess $w$ for the corresponding $z$ into one or more CNFs, conjunctions (AND) of clauses or disjunctions (OR) of Boolean variables. The basic encoding is very simple: one variable for each entry of the $m\times n$ $(0,1)$-matrix $A$, one clause for each and every $a\times b$ minor in rows $r_1,\dots,r_a$ and columns $c_1,\dots,c_b$
	\begin{equation*}
		\bigvee_{i=1}^a\bigvee_{j=1}^b\neg a_{ij}
	\end{equation*}
	and a cardinality constraint requiring $A$ to have exactly $w$ ones (its encoding details are discussed below). Any solution to this CNF forms an admissible matrix, proving $z\ge w$; conversely if the instance is unsatisfiable (UNSAT) this indicates $z<w$. Most SAT solvers have an option to output a concrete, machine-verifiable UNSAT proof if the instance turns out that way \cite{drattrim}.
	
	To this basic scheme we add some major optimisations, without which extending the range of known Zarankiewicz function values would not be possible.
	
	\subsection{Generating partitions}
	
	An admissible or maximal matrix clearly remains as such under all row and column permutations. It is therefore enough for a given $w$ to solve instances where the row and column sums are \textit{fixed}, over all possible combinations of \textit{unordered} row and column partitions not forbidden by the arguments of section \ref{sec:eba} -- an approach very much like Heule's cube-and-conquer paradigm \cite{cubeandconquer}. To generate all such partitions efficiently we use algorithm \ref{alg:admpart}.
	
	\begin{algorithm}
		\caption{Admissible (by arguments A and I) column partition generator}\label{alg:admpart}
		\begin{algorithmic}[1]
			\State $p\gets$ empty stack \Comment{workspace for building up partitions}
			\Procedure{P}{$a$, $b$, $m$, $n$, $w$}
			\State $L_A\gets(b-1)\binom ma$ \Comment{only set at procedure start, immutable afterwards}
			\If{$\sum_i\binom{p_i}a>L_A$ or $\sum p>z_{a,b}(m,|p|)$} \Comment{$|p|$ is the current length of $p$}
			\State \textbf{return}
			\ElsIf{$w=0$}
			\State \textbf{output the contents of} $p$
			\ElsIf{$k>(m-1)n$} \Comment{by the pigeonhole principle, some further columns must sum to $m$}
			\State $d\gets k-(m-1)n$
			\State push $m$ $d$ times onto $p$
			\State $\operatorname{P}(a,\,b,\,m,\,n-d,\,w-dm)$
			\State pop $d$ times from $p$
			\Else
			\For{$t\in[\lceil w/n\rceil,\min(w,m)]$} \Comment{all possible values for the next part}
			\State push $t$ onto $p$
			\State $\operatorname{P}(a,\,b,\,t,\,n-1,\,w-t)$
			\State pop from $p$
			\EndFor
			\EndIf
			\EndProcedure
		\end{algorithmic}
	\end{algorithm}
	\begin{theorem}
		Algorithm \ref{alg:admpart} generates all admissible partitions for an $m\times n$ matrix, $a\times b$ minor and $w$ ones in lexicographic order -- partitions of $w$ into $n$ parts in $[0,m]$ -- with the parts in each partition listed in non-increasing order.
	\end{theorem}
	\begin{proof}
		Ignoring lines 4 and 8--12 for now, the recursive call to $\operatorname P$ in line 16 specifies an upper part limit of the last (topmost) element $t$ of the stack $p$, so part sizes do not increase from left to right. By the pigeonhole principle the largest part of a partition of $w$ into $n$ parts is at least $\lceil w/n\rceil$, so this is the lower bound for $t$; the upper bound of $\min(w,m)$ is trivial. Because $t$ is varied through all its possible values in increasing order at every point in the recursion tree, the partitions are output in lexicographic order.
		
		Lines 8--12 avoid unnecessary recursive calls to $\operatorname P$ when there is only one possible value for $t$. The admissibility checks in line 4 depend on the non-increasing partition ordering, which in turn ensures that the first $n'$ column sums in $p$ for any $n'<n$ are the most pessimal choice for the column sums of an $m\times n'$ minor of the $m\times n$ matrix; if this minor partition is admissible then all other $m\times n'$ minor partitions in $p$ are admissible because $\sum_i\binom{p_i}a$ for argument A and $\sum p>z_{a,b}(m,|p|)$ for argument I cannot be higher for the other partitions.
		
		Because line 4 is executed in every call to $\operatorname P$, branches of the recursion tree leading to only inadmissible partitions are pruned as soon as possible.
	\end{proof}
	The algorithm to generate row partitions is similar. Once all possible row and column partitions have been obtained argument D can then be used to remove partition pairs (considering the row with the most, $r$, ones and the $r$ columns with the least ones -- if argument D fails for this most pessimal column choice it must also fail for all other column choices -- and vice versa).
	
	\subsection{Cardinality constraints}
	
	To express that exactly $k$ out of $n$ bits $b_1,\dots,b_n$ should be true we use the equality variant of Sinz's sequential counter encoding \cite{sinz} as described, tested and deemed fastest for general use among different cardinality constraint encodings by Wynn \cite{cardconstraint}. $k(n-k)$ auxiliary variables $a_{i,j}$ are used where $1\le i\le k$ and $1\le j\le n-k$, with the following clauses (all literals $a_{i,j}$ with $i$ or $j$ outside their specified ranges are dropped):
	\begin{gather*}
		\bigwedge_{i=1}^k\bigwedge_{j=1}^{n-k-1}\neg a_{i,j}\lor a_{i,j+1}\qquad\bigwedge_{i=0}^k\bigwedge_{j=1}^{n-k}\neg a_{i,j}\lor a_{i+1,j}\lor\neg b_{i+j}\\
		\bigwedge_{i=1}^{k-1}\bigwedge_{j=1}^{n-k}a_{i,j}\lor\neg a_{i+1,j}\qquad\bigwedge_{i=1}^k\bigwedge_{j=0}^{n-k}a_{i,j}\lor\neg a_{i,j+1}\lor b_{i+j}
	\end{gather*}
	
	This encoding has two desirable properties:
	\begin{itemize}
		\item If a partial assignment of the $b_i$ is such that said assignment cannot be completed without violating the cardinality constraint, unit propagation alone will lead to a contradiction (empty clause).
		\item If exactly $k$ of the $b_i$ are assigned true, unit propagation alone will assign the other $b_i$ false.
	\end{itemize}
	
	Since unit propagation is hardwired into all state-of-the-art SAT solvers, using the above encoding should result in faster rejection of partially filled matrices that cannot be completed to an admissible matrix.
	
	\subsection{Lexicographic constraints}
	
	\begin{figure}
		\centering
		\begin{equation*}
			\begin{bmatrix}
				0&1&1&1&1\\
				1&1&0&1&0\\
				0&0&0&0&1\\
				0&1&0&1&1\\
				0&1&0&0&0
			\end{bmatrix}\xrightarrow{\text{sort rows}}
			\begin{bmatrix}
				1&1&0&1&0\\
				0&1&1&1&1\\
				0&1&0&1&1\\
				0&1&0&0&0\\
				0&0&0&0&1
			\end{bmatrix}\xrightarrow{\text{sort columns}}
			\begin{bmatrix}
				1&1&0&1&0\\
				0&1&1&1&1\\
				0&1&0&1&1\\
				0&1&0&0&0\\
				0&0&0&0&1
			\end{bmatrix}
		\end{equation*}
		\caption{Reverse-lexicographically sorting a $(0,1)$-matrix to a fixed point.}
		\label{fig:lexsort}
	\end{figure}
	
	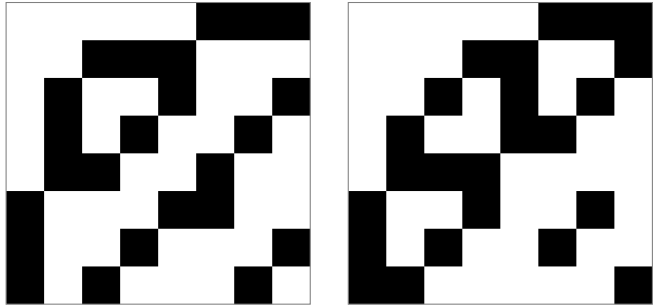
\begin{figure}
		\centering
		\begin{tikzpicture}
			\begin{scope}[scale=0.5]
				\fill (0,3) rectangle (1,2) (0,2) rectangle (1,1) (0,1) rectangle (1,0) (1,6) rectangle (2,5) (1,5) rectangle (2,4) (1,4) rectangle (2,3) (2,7) rectangle (3,6) (2,4) rectangle (3,3) (2,1) rectangle (3,0) (3,7) rectangle (4,6) (3,5) rectangle (4,4) (3,2) rectangle (4,1) (4,7) rectangle (5,6) (4,6) rectangle (5,5) (4,3) rectangle (5,2) (5,8) rectangle (6,7) (5,4) rectangle (6,3) (5,3) rectangle (6,2) (6,8) rectangle (7,7) (6,5) rectangle (7,4) (6,1) rectangle (7,0) (7,8) rectangle (8,7) (7,6) rectangle (8,5) (7,2) rectangle (8,1);
				\draw[black!50] (0,0) rectangle (8,8);
			\end{scope}
			\begin{scope}[scale=0.5,xshift=9cm]
				\fill (0,3) rectangle (1,2) (0,2) rectangle (1,1) (0,1) rectangle (1,0) (1,5) rectangle (2,4) (1,4) rectangle (2,3) (1,1) rectangle (2,0) (2,6) rectangle (3,5) (2,4) rectangle (3,3) (2,2) rectangle (3,1) (3,7) rectangle (4,6) (3,4) rectangle (4,3) (3,3) rectangle (4,2) (4,7) rectangle (5,6) (4,6) rectangle (5,5) (4,5) rectangle (5,4) (5,8) rectangle (6,7) (5,5) rectangle (6,4) (5,2) rectangle (6,1) (6,8) rectangle (7,7) (6,6) rectangle (7,5) (6,3) rectangle (7,2) (7,8) rectangle (8,7) (7,7) rectangle (8,6) (7,1) rectangle (8,0);
				\draw[black!50] (0,0) rectangle (8,8);
			\end{scope}
		\end{tikzpicture}
		\caption{Two non-identical yet isomorphic maximal matrices (for $a=b=2$, $m=n=8$) that satisfy all constraints in section \ref{sec:method}.}
		\label{fig:isomats}
	\end{figure}
	
	Even with fixed row and column sums, there still remain the symmetries of swapping two rows or two columns with the \textit{same} sum. These symmetries are broken by requiring groups of rows or columns with the same sum to be contiguous and lexicographically sorted; every $(0,1)$-matrix can be permuted to satisfy this property by the following theorem.
	
	\begin{theorem}
		Lexicographically sorting rows and columns of any $(0,1)$-matrix $A$ alternately as in Figure \ref{fig:lexsort} will reach a fixed point (both rows and columns sorted) in a finite number of steps. This remains true even if the sets of rows and columns are partitioned so that rows and columns cannot move across partitions.
	\end{theorem}
	
	\begin{proof}
		With rows and columns indexed starting from 0, define $f(A)=\sum_i\sum_j2^{i+j}a_{ij}$. Swapping rows/columns $a$ and $b$ where $a<b$ but the numerical value $n_a$ of column $a$ is greater than $n_b$ changes $f(A)$ by $2^bn_a+2^an_b-2^an_a-2^bn_b=(2^b-2^a)(n_a-n_b)>0$, i.e.\ sorting two out-of-order rows/columns strictly increases (or decreases, if sorting in reverse order) $f(A)$, which is clearly integral and bounded by 0 from below and $\sum_i\sum_j2^{i+j}$ from above. Since there are a finite number of possibilities for each value in the strictly monotone sequence of $f(A)$'s generated, it must terminate at a point when $A$ is sorted \textit{both} in rows and columns.
	\end{proof}
	
	Given two equal-length strings of Boolean variables $a_1,\dots,a_n$ and $b_1,\dots,b_n$, the binary number represented by the $a_i$ may be constrained to be \textit{at most} that represented by the $b_i$ (where $a_1,b_1$ are most significant) through $n-1$ auxiliary variables $c_1,\dots,c_{n-1}$ and the clauses ($c_0$ and $c_n$ are dropped)
	\begin{gather*}
		\bigwedge_{i=1}^{n-2}\neg c_i\lor c_{i+1}\\
		\bigwedge_{i=1}^n c_{i-1}\lor\neg a_i\lor b_i\qquad\bigwedge_{i=1}^n c_{i-1}\lor a_i\lor b_i\lor\neg c_i\\
		\bigwedge_{i=1}^n c_{i-1}\lor\neg a_i\lor\neg b_i\lor\neg c_i\qquad\bigwedge_{i=1}^n c_{i-1}\lor a_i\lor\neg b_i\lor c_i
	\end{gather*}
	
	In our application of this form of symmetry breaking to the problem at hand the sort order is reversed: 1 comes before 0. The cardinality and lexicographic constraints do not remove all symmetries of a $(0,1)$-matrix (see figure \ref{fig:isomats}) -- doing so would require solving the graph isomorphism problem -- but they are nevertheless very useful in reducing the number of instance solutions.
	
	\subsection{Software}
	
	All SAT solving was done with Kissat \cite{kissat} on one laptop computer with the \texttt{--sat} and \texttt{--unsat} flags set according to whether or not a solution was expected, and no other settings touched. The maximal matrices in the $m=n$ case were filtered to remove isomorphs using the \texttt{shortg} utility in nauty \cite{nauty}; the automorphism groups of the corresponding bipartite graphs were computed using GAP (\url{https://gap-system.org}).
	
	The partitioning and CNF-building code written for this project, together with the raw results obtained, is available in our Kyoto repository \cite{kyoto}.
	
	\section{Tables for the Zarankiewicz function}\label{sec:tables}
	
	The following three tables are corrected and extended versions of the tables for $z_a(m,n)$ given in Guy \cite{guy} where $a=2,3,4$. Values above solid lines are both exact and given by theorem \ref{thm:packlimit}; the dashed lines indicate the limits of Guy's tables and grey backgrounds indicate errors Guy made. A bold value is exact, proven by the methods in this paper; other values are the upper bounds given by theorem \ref{thm:packlimit}.
	
	\subsection{Discussion}
	
	The last section of Héger's thesis \cite{heger} is devoted to proving exact values and tighter bounds for $z_2(m,n)$, which is closely related to finite geometries by Reiman's construction \cite{reiman}: the point-line incidence matrix of the projective plane of prime power order $q$ furnishes a maximal matrix for $z_2(q^2+q+1)$, showing that it is equal to $(q+1)(q^2+q+1)$. Héger collects the new results into another table for $z_2(m,n)$, which has some values marked exact that are not marked as such in table \ref{table:2x2}, but that table comes with a caveat:
	\begin{quotation}
		In some cases we did rely on the exact values reported by Guy. Possibly undiscovered inaccuracies there may result in inaccurate values here as well.
	\end{quotation}
	
	By computing the $z$-values through an independent method we have completed the list of errors in Guy's $z_a(m,n)$ tables -- there are only eight such errors, all in the $z_2(m,n)$ table, and all are too low by just one. There are meanwhile no discrepancies in the values marked as exact in both Héger's table and table \ref{table:2x2}, where Héger did not rely on Guy.
	
	The link to finite geometries does not carry over to larger minor sizes, where the bound of theorem \ref{thm:packlimit} appears to be less sharp, particularly when $m\approx n$. Better bounds at the edges of the exact region can often be derived by applying the arguments of section \ref{sec:eba} to eliminate all possible partitions and hence the need for any SAT solving; this gives for example $z_4(11,14)\le106$.
	
	\begin{landscape}
		\begin{table}
			\centering
			\begin{tikzpicture}
				\matrix(A)[matrix of nodes,every cell/.style={anchor=base east}] {
					$m\setminus n$&2&3&4&5&6&7&8&9&10&11&12&13&14&15&16&17&18&19&20&21&22&23&24&25&26&27&28&29&30&31&32&33&34&35\\
					\hline
					2&\textbf{3}&\textbf{4}&\textbf{5}&\textbf{6}&\textbf{7}&\textbf{8}&\textbf{9}&\textbf{10}&\textbf{11}&\textbf{12}&\textbf{13}&\textbf{14}&\textbf{15}&\textbf{16}&\textbf{17}&\textbf{18}&\textbf{19}&\textbf{20}&\textbf{21}&\textbf{22}&\textbf{23}&\textbf{24}&\textbf{25}&\textbf{26}&\textbf{27}&\textbf{28}&\textbf{29}&\textbf{30}&\textbf{31}&\textbf{32}&\textbf{33}&\textbf{34}&\textbf{35}&\textbf{36}\\
					3&&\textbf{6}&\textbf{7}&\textbf{8}&\textbf{9}&\textbf{10}&\textbf{11}&\textbf{12}&\textbf{13}&\textbf{14}&\textbf{15}&\textbf{16}&\textbf{17}&\textbf{18}&\textbf{19}&\textbf{20}&\textbf{21}&\textbf{22}&\textbf{23}&\textbf{24}&\textbf{25}&\textbf{26}&\textbf{27}&\textbf{28}&\textbf{29}&\textbf{30}&\textbf{31}&\textbf{32}&\textbf{33}&\textbf{34}&\textbf{35}&\textbf{36}&\textbf{37}&\textbf{38}\\
					4&&&\textbf{9}&\textbf{10}&\textbf{12}&\textbf{13}&\textbf{14}&\textbf{15}&\textbf{16}&\textbf{17}&\textbf{18}&\textbf{19}&\textbf{20}&\textbf{21}&\textbf{22}&\textbf{23}&\textbf{24}&\textbf{25}&\textbf{26}&\textbf{27}&\textbf{28}&\textbf{29}&\textbf{30}&\textbf{31}&\textbf{32}&\textbf{33}&\textbf{34}&\textbf{35}&\textbf{36}&\textbf{37}&\textbf{38}&\textbf{39}&\textbf{40}&\textbf{41}\\
					5&&&&\textbf{12}&\textbf{14}&\textbf{15}&\textbf{17}&\textbf{18}&\textbf{20}&\textbf{21}&\textbf{22}&\textbf{23}&\textbf{24}&\textbf{25}&\textbf{26}&\textbf{27}&\textbf{28}&\textbf{29}&\textbf{30}&\textbf{31}&\textbf{32}&\textbf{33}&\textbf{34}&\textbf{35}&\textbf{36}&\textbf{37}&\textbf{38}&\textbf{39}&\textbf{40}&\textbf{41}&\textbf{42}&\textbf{43}&\textbf{44}&\textbf{45}\\
					6&&&&&\textbf{16}&\textbf{18}&\textbf{19}&\textbf{21}&\textbf{22}&\textbf{24}&\textbf{25}&\textbf{27}&\textbf{28}&\textbf{30}&\textbf{31}&\textbf{32}&\textbf{33}&\textbf{34}&\textbf{35}&\textbf{36}&\textbf{37}&\textbf{38}&\textbf{39}&\textbf{40}&\textbf{41}&\textbf{42}&\textbf{43}&\textbf{44}&\textbf{45}&\textbf{46}&\textbf{47}&\textbf{48}&\textbf{49}&\textbf{50}\\
					7&&&&&&\textbf{21}&\textbf{22}&\textbf{24}&\textbf{25}&\textbf{27}&\textbf{28}&\textbf{30}&\textbf{31}&\textbf{33}&\textbf{34}&\textbf{36}&\textbf{37}&\textbf{39}&\textbf{40}&\textbf{42}&\textbf{43}&\textbf{44}&\textbf{45}&\textbf{46}&\textbf{47}&\textbf{48}&\textbf{49}&\textbf{50}&\textbf{51}&\textbf{52}&\textbf{53}&\textbf{54}&\textbf{55}&\textbf{56}\\
					8&&&&&&&\textbf{24}&\textbf{26}&\textbf{28}&\textbf{30}&\textbf{32}&\textbf{33}&\textbf{35}&\textbf{36}&\textbf{38}&\textbf{39}&\textbf{41}&\textbf{42}&\textbf{44}&\textbf{45}&\textbf{47}&\textbf{48}&\textbf{50}&\textbf{51}&\textbf{53}&\textbf{54}&\textbf{56}&\textbf{57}&\textbf{58}&\textbf{59}&\textbf{60}&\textbf{61}&\textbf{62}&\textbf{63}\\
					9&&&&&&&&\textbf{29}&\textbf{31}&\textbf{33}&\textbf{36}&\textbf{37}&\textbf{39}&\textbf{40}&\textbf{42}&\textbf{43}&\textbf{45}&\textbf{46}&\textbf{48}&\textbf{49}&\textbf{51}&\textbf{52}&\textbf{54}&\textbf{55}&\textbf{57}&\textbf{58}&\textbf{60}&\textbf{61}&\textbf{63}&\textbf{64}&\textbf{66}&\textbf{67}&\textbf{69}&\textbf{70}\\
					10&&&&&&&&&\textbf{34}&\textbf{36}&\textbf{39}&\textbf{40}&\textbf{42}&\textbf{44}&\textbf{46}&\textbf{47}&\textbf{49}&\textbf{51}&\textbf{52}&\textbf{54}&\textbf{55}&\textbf{57}&\textbf{58}&\textbf{60}&\textbf{61}&\textbf{63}&\textbf{64}&\textbf{66}&\textbf{67}&\textbf{69}&\textbf{70}&\textbf{72}&\textbf{73}&\textbf{75}\\
					11&&&&&&&&&&\textbf{39}&\textbf{42}&\textbf{44}&\textbf{45}&\textbf{47}&\textbf{50}&\textbf{51}&\textbf{53}&\textbf{55}&\textbf{57}&\textbf{59}&\textbf{60}&\textbf{62}&\textbf{63}&\textbf{65}&\textbf{66}&\textbf{68}&\textbf{69}&\textbf{71}&\textbf{72}&\textbf{74}&\textbf{75}&\textbf{77}&\textbf{78}&\textbf{80}\\
					12&&&&&&&&&&&\textbf{45}&\textbf{48}&\textbf{49}&\textbf{51}&\textbf{53}&\textbf{55}&\textbf{57}&\textbf{60}&\textbf{61}&\textbf{63}&\textbf{65}&\textbf{66}&\textbf{68}&\textbf{70}&\textbf{72}&\textbf{73}&\textbf{75}&\textbf{76}&\textbf{78}&\textbf{79}&\textbf{81}&\textbf{82}&\textbf{84}&\textbf{85}\\
					13&&&&&&&&&&&&\textbf{52}&\textbf{53}&\textbf{55}&\textbf{57}&\textbf{59}&\textbf{61}&\textbf{64}&\textbf{66}&\textbf{67}&\textbf{69}&\textbf{71}&\textbf{73}&\textbf{75}&\textbf{78}&\textbf{79}&\textbf{81}&\textbf{82}&\textbf{84}&\textbf{85}&\textbf{87}&\textbf{88}&\textbf{90}&\textbf{91}\\
					14&&&&&&&&&&&&&\textbf{56}&\textbf{58}&\textbf{60}&\textbf{63}&\textbf{65}&\textbf{68}&\textbf{70}&\textbf{72}&\textbf{73}&\textbf{75}&|[fill=black!25]|\textbf{78}&|[fill=black!25]|\textbf{80}&|[fill=black!25]|\textbf{82}&|[fill=black!25]|\textbf{84}&|[fill=black!25]|\textbf{86}&\textbf{87}&\textbf{89}&\textbf{91}&\textbf{92}&\textbf{94}&\textbf{96}&\textbf{98}\\
					15&&&&&&&&&&&&&&|[fill=black!25]|\textbf{61}&|[fill=black!25]|\textbf{64}&|[fill=black!25]|\textbf{67}&\textbf{69}&\textbf{72}&\textbf{75}&\textbf{77}&\textbf{78}&\textbf{80}&\textbf{82}&\textbf{85}&\textbf{86}&\textbf{88}&\textbf{91}&\textbf{93}&\textbf{95}&\textbf{96}&\textbf{98}&\textbf{100}&\textbf{102}&\textbf{105}\\
					16&&&&&&&&&&&&&&&\textbf{67}&\textbf{70}&\textbf{73}&\textbf{76}&\textbf{80}&\textbf{81}&\textbf{83}&\textbf{85}&\textbf{87}&\textbf{90}&\textbf{91}&\textbf{93}&\textbf{96}&\textbf{98}&\textbf{100}&\textbf{102}&\textbf{103}&106&108&110\\
					17&&&&&&&&&&&&&&&&\textbf{74}&\textbf{77}&\textbf{80}&\textbf{84}&\textbf{85}&\textbf{87}&\textbf{89}&\textbf{91}&\textbf{94}&\textbf{96}&\textbf{98}&\textbf{101}&\textbf{102}&105&107&109&111&113&115\\
					18&&&&&&&&&&&&&&&&&\textbf{81}&\textbf{84}&\textbf{88}&\textbf{90}&\textbf{91}&\textbf{93}&\textbf{96}&\textbf{99}&\textbf{101}&\textbf{103}&107&109&111&113&115&117&119&121\\
					19&&&&&&&&&&&&&&&&&&\textbf{88}&\textbf{92}&\textbf{95}&\textbf{96}&\textbf{98}&\textbf{100}&\textbf{103}&\textbf{106}&110&112&115&117&119&121&123&125&127\\
					20&&&&&&&&&&&&&&&&&&&\textbf{96}&\textbf{100}&\textbf{101}&\textbf{103}&\textbf{105}&\textbf{108}&\textbf{111}&115&117&120&122&125&127&129&131&133\\
					21&&&&&&&&&&&&&&&&&&&&\textbf{105}&\textbf{106}&\textbf{108}&\textbf{110}&115&117&120&122&125&127&130&132&135&137&140\\
					22&&&&&&&&&&&&&&&&&&&&&\textbf{108}&\textbf{110}&\textbf{114}&120&122&125&127&130&132&135&137&140&142&145\\
					23&&&&&&&&&&&&&&&&&&&&&&\textbf{115}&\textbf{118}&125&128&130&133&135&138&140&143&145&148&150\\
					24&&&&&&&&&&&&&&&&&&&&&&&\textbf{122}&130&133&136&139&141&144&146&149&151&154&156\\
				};
				\draw (A-1-1.north east) -- (A-24-1.south east);
				\draw (A-8-8.south west) -- (A-8-8.north west) -- (A-8-10.north east) -- (A-8-10.south east) -- (A-8-11.south east) -- (A-9-11.south east) -- (A-9-17.south east) -- (A-10-17.south east) -- (A-10-19.south east) -- (A-11-19.south east) -- (A-11-24.south east) -- (A-12-24.south east) -- (A-12-25.south east) -- (A-13-25.south east) -- (A-13-33.south east) -- (A-14-33.south east) -- (A-14-34.south east) -- (A-15-34.south east) -- (A-15-35.south east);
				\draw[dashed] (A-16-16.south west) -- (A-16-16.north west) -- (A-16-19.north east) -- (A-16-19.south east) -- (A-16-20.south east) -- (A-16-20.north east) -- (A-16-21.north east) -- (A-14-21.south east) -- (A-14-28.south east) -- (A-13-28.south east);
			\end{tikzpicture}
			\caption{$z_2(m,n)$}
			\label{table:2x2}
		\end{table}
	\end{landscape}
	
	\newpage
	
	\begin{table}[h]
		\centering
		\begin{tikzpicture}
			\matrix(A)[matrix of nodes,every cell/.style={anchor=base east}] {
				$m\setminus n$&3&4&5&6&7&8&9&10&11&12&13&14&15&16&17&18&19&20&21&22&23\\
				\hline
				3&\textbf{8}&\textbf{10}&\textbf{12}&\textbf{14}&\textbf{16}&\textbf{18}&\textbf{20}&\textbf{22}&\textbf{24}&\textbf{26}&\textbf{28}&\textbf{30}&\textbf{32}&\textbf{34}&\textbf{36}&\textbf{38}&\textbf{40}&\textbf{42}&\textbf{44}&\textbf{46}&\textbf{48}\\
				4&&\textbf{13}&\textbf{16}&\textbf{18}&\textbf{21}&\textbf{24}&\textbf{26}&\textbf{28}&\textbf{30}&\textbf{32}&\textbf{34}&\textbf{36}&\textbf{38}&\textbf{40}&\textbf{42}&\textbf{44}&\textbf{46}&\textbf{48}&\textbf{50}&\textbf{52}&\textbf{54}\\
				5&&&\textbf{20}&\textbf{22}&\textbf{25}&\textbf{28}&\textbf{30}&\textbf{33}&\textbf{36}&\textbf{38}&\textbf{41}&\textbf{44}&\textbf{46}&\textbf{49}&\textbf{52}&\textbf{54}&\textbf{57}&\textbf{60}&\textbf{62}&\textbf{64}&\textbf{66}\\
				6&&&&\textbf{26}&\textbf{29}&\textbf{32}&\textbf{36}&\textbf{39}&\textbf{42}&\textbf{45}&\textbf{48}&\textbf{50}&\textbf{53}&\textbf{56}&\textbf{58}&\textbf{61}&\textbf{64}&\textbf{66}&\textbf{69}&\textbf{72}&\textbf{74}\\
				7&&&&&\textbf{33}&\textbf{37}&\textbf{40}&\textbf{44}&\textbf{47}&\textbf{50}&\textbf{53}&\textbf{56}&\textbf{60}&\textbf{63}&\textbf{66}&\textbf{69}&\textbf{72}&\textbf{75}&\textbf{78}&\textbf{81}&\textbf{84}\\
				8&&&&&&\textbf{42}&\textbf{45}&\textbf{50}&\textbf{53}&\textbf{57}&\textbf{60}&\textbf{64}&\textbf{67}&\textbf{70}&\textbf{74}&\textbf{77}&\textbf{81}&\textbf{84}&\textbf{87}&\textbf{90}&\textbf{94}\\
				9&&&&&&&\textbf{49}&\textbf{54}&\textbf{59}&\textbf{64}&\textbf{67}&\textbf{70}&\textbf{73}&\textbf{77}&\textbf{81}&\textbf{85}&\textbf{89}&\textbf{93}&\textbf{96}&\textbf{100}&104\\
				10&&&&&&&&\textbf{60}&\textbf{64}&\textbf{68}&\textbf{73}&\textbf{77}&\textbf{81}&\textbf{85}&\textbf{90}&\textbf{94}&\textbf{98}&\textbf{102}&108&112&116\\
				11&&&&&&&&&\textbf{69}&\textbf{74}&\textbf{80}&\textbf{84}&\textbf{88}&\textbf{92}&\textbf{96}&\textbf{101}&109&113&117&121&125\\
				12&&&&&&&&&&\textbf{80}&\textbf{86}&\textbf{91}&\textbf{96}&\textbf{99}&108&113&118&122&127&132&136\\
				13&&&&&&&&&&&\textbf{92}&\textbf{98}&\textbf{104}&\textbf{107}&117&122&126&131&136&140&145\\
				14&&&&&&&&&&&&\textbf{105}&\textbf{112}&\textbf{115}&125&130&136&141&146&151&155\\
				15&&&&&&&&&&&&&\textbf{120}&\textbf{123}&134&139&144&150&155&160&166\\
				16&&&&&&&&&&&&&&\textbf{128}&142&148&154&160&165&170&176\\
			};
			\draw (A-1-1.north east) -- (A-15-1.south east);
			\draw (A-5-5.south west) -- (A-5-5.north west) -- (A-5-9.north east) -- (A-5-9.south east) -- (A-5-21.south east) -- (A-6-21.south east) -- (A-6-22.south east);
			\draw[dashed] (A-9-9.south east) -- (A-8-9.south east) -- (A-8-11.south east) -- (A-7-11.south east) -- (A-7-13.south east) -- (A-6-13.south east) -- (A-6-21.south east);
		\end{tikzpicture}
		\caption{$z_3(m,n)$}
		\label{table:3x3}
	\end{table}
	
	\begin{table}[h]
		\centering
		\begin{tikzpicture}
			\matrix(A)[matrix of nodes,every cell/.style={anchor=base east}] {
				$m\setminus n$&4&5&6&7&8&9&10&11&12&13&14&15&16&17&18&19&20&21\\
				\hline
				4&\textbf{15}&\textbf{18}&\textbf{21}&\textbf{24}&\textbf{27}&\textbf{30}&\textbf{33}&\textbf{36}&\textbf{39}&\textbf{42}&\textbf{45}&\textbf{48}&\textbf{51}&\textbf{54}&\textbf{57}&\textbf{60}&\textbf{63}&\textbf{66}\\
				5&&\textbf{22}&\textbf{26}&\textbf{30}&\textbf{33}&\textbf{37}&\textbf{41}&\textbf{45}&\textbf{48}&\textbf{52}&\textbf{56}&\textbf{60}&\textbf{63}&\textbf{66}&\textbf{69}&\textbf{72}&\textbf{75}&\textbf{78}\\
				6&&&\textbf{31}&\textbf{36}&\textbf{39}&\textbf{43}&\textbf{47}&\textbf{51}&\textbf{55}&\textbf{59}&\textbf{63}&\textbf{67}&\textbf{71}&\textbf{75}&\textbf{78}&\textbf{82}&\textbf{86}&\textbf{90}\\
				7&&&&\textbf{42}&\textbf{45}&\textbf{49}&\textbf{54}&\textbf{58}&\textbf{63}&\textbf{68}&\textbf{72}&\textbf{77}&\textbf{82}&\textbf{87}&\textbf{90}&\textbf{95}&\textbf{100}&\textbf{105}\\
				8&&&&&\textbf{51}&\textbf{55}&\textbf{60}&\textbf{65}&\textbf{70}&\textbf{75}&\textbf{80}&\textbf{85}&\textbf{90}&\textbf{95}&\textbf{99}&106&111&115\\
				9&&&&&&\textbf{61}&\textbf{67}&\textbf{72}&\textbf{78}&\textbf{84}&\textbf{88}&\textbf{94}&\textbf{99}&108&113&118&123&129\\
				10&&&&&&&\textbf{74}&\textbf{79}&\textbf{86}&\textbf{93}&\textbf{97}&108&114&120&126&131&136&141\\
				11&&&&&&&&\textbf{86}&\textbf{93}&\textbf{100}&112&118&124&130&136&142&148&154\\
				12&&&&&&&&&\textbf{100}&\textbf{108}&121&127&134&141&148&154&161&168\\
				13&&&&&&&&&&\textbf{117}&130&138&145&153&159&166&173&180\\
			};
			\draw (A-1-1.north east) -- (A-11-1.south east);
			\draw (A-4-4.south west) -- (A-4-4.north west) -- (A-4-11.north east) -- (A-4-11.south east) -- (A-4-18.south east) -- (A-5-18.south east) -- (A-5-19.south east);
			\draw[dashed] (A-7-7.south west) -- (A-7-7.north west) -- (A-7-8.north east) -- (A-5-8.south east) -- (A-5-18.south east);
		\end{tikzpicture}
		\caption{$z_4(m,n)$}
		\label{table:4x4}
	\end{table}
	
	\section{Maximal square matrices}
	
	Each row in the table in this section contains a value for $a$, a value for $m$ and all maximal matrices for $z_a(m)$ up to isomorphism (which includes transposing the matrix) together with their row and column sums and automorphism groups. The matrices are presented both as images and as coded strings that can be decoded through the \texttt{decode\_array()} function in Kyoto \cite{kyoto}.
	
	A $(0,1)$-matrix is encoded by flattening it so that rows remain contiguous, padding the result on the right to a multiple of 8 bits with zeros, interpreting each byte in \textit{little-endian} order and encoding the final byte sequence using Base64. The height and width are prepended, separated by spaces.
	
	Where possible a symmetric presentation of each matrix has been chosen; those matrices without such representations have their encodings marked with an asterisk.
	
	\newcommand{\wrap}[1]{\parbox[b]{.4\linewidth}{\vspace{1mm}#1}}
\begin{longtable}[c]{cccl}
		$(a,m)$ & $z_a(m)$ & Maximal matrix & Code, row/column sums, automorphism group\\
		\midrule
		\endfirsthead
		$(a,m)$ & $z_a(m)$ & Maximal matrix & Code, row/column sums, automorphism group\\
		\midrule
		\endhead
		\endfoot
		\endlastfoot
		$(2,2)$ & 3 & \begin{tikzpicture}\begin{scope}[scale=0.5]\fill (0,2) rectangle (1,1) (1,2) rectangle (2,1) (0,1) rectangle (1,0); \draw[black!50] (0,0) rectangle (2,2);\end{scope}\end{tikzpicture} & \wrap{\texttt{2 2 Bw==}\\$(2,1)$ $(2,1)$\\$C_2$, order 2}\\
		\midrule
		$(2,3)$ & 6 & \begin{tikzpicture}\begin{scope}[scale=0.5]\fill (0,3) rectangle (1,2) (1,3) rectangle (2,2) (0,2) rectangle (1,1) (2,2) rectangle (3,1) (1,1) rectangle (2,0) (2,1) rectangle (3,0); \draw[black!50] (0,0) rectangle (3,3);\end{scope}\end{tikzpicture} & \wrap{\texttt{3 3 qwE=}\\$(2^3)$ $(2^3)$\\$D_6$, order 12} \\
		\midrule
		$(2,4)$ & 9 & \begin{tikzpicture}\begin{scope}[scale=0.5]\fill (1,4) rectangle (2,3) (2,4) rectangle (3,3) (3,4) rectangle (4,3) (0,3) rectangle (1,2) (1,3) rectangle (2,2) (0,2) rectangle (1,1) (2,2) rectangle (3,1) (0,1) rectangle (1,0) (3,1) rectangle (4,0); \draw[black!50] (0,0) rectangle (4,4);\end{scope}\end{tikzpicture} & \wrap{\texttt{4 4 PpU=}\\$(3,2^3)$ $(3,2^3)$\\$S_3\times C_2$, order 12} \\
		\midrule
		$(2,5)$ & 12 & \begin{tikzpicture}\begin{scope}[scale=0.5]\fill (0,5) rectangle (1,4) (2,5) rectangle (3,4) (3,5) rectangle (4,4) (1,4) rectangle (2,3) (2,4) rectangle (3,3) (4,4) rectangle (5,3) (0,3) rectangle (1,2) (1,3) rectangle (2,2) (0,2) rectangle (1,1) (4,2) rectangle (5,1) (1,1) rectangle (2,0) (3,1) rectangle (4,0); \draw[black!50] (0,0) rectangle (5,5);\end{scope}\end{tikzpicture} & \wrap{\texttt{5 5 zY6oAA==}\\$(3^2,2^3)$ $(3^2,2^3)$\\$C_2^2$, order 4} \\
		&& \begin{tikzpicture}\begin{scope}[scale=0.5]\fill (1,5) rectangle (2,4) (2,5) rectangle (3,4) (3,5) rectangle (4,4) (4,5) rectangle (5,4) (0,4) rectangle (1,3) (1,4) rectangle (2,3) (0,3) rectangle (1,2) (2,3) rectangle (3,2) (0,2) rectangle (1,1) (3,2) rectangle (4,1) (0,1) rectangle (1,0) (4,1) rectangle (5,0); \draw[black!50] (0,0) rectangle (5,5);\end{scope}\end{tikzpicture} & \wrap{\texttt{5 5 fpQUAQ==}\\$(4,2^4)$ $(4,2^4)$\\$S_4\times C_2$, order 48} \\
\end{longtable}
	
	\subsection{Discussion}
	
	Nearly all of the found maximal matrices can be arranged to be symmetric, with all exceptions at least having some other automorphisms. This suggests that exhaustively checking matrices for all-one minors as Collins \cite{collins} did is highly unlikely to yield maximal matrices, even if the tools used can support an exhaustive search at the desired matrix size. Instead, to obtain explicit lower bounds for the Zarankiewicz function, one should try extending smaller matrices to larger ones by adding as many ones as possible.
	
	The first non-trivial maximal matrices for a given $a$ follow a simple pattern: for $a\le m<2a$ the complement of the bipartite graph equivalent to the maximal matrix is simply $2(m-a)+1$ isolated edges, and when $m=2a$ the complement is $a-1$ isolated edges and a $2(a+1)$-cycle. Yang \cite{yang} has proved that these are indeed the unique maximal matrices up to isomorphism for these sets of parameters, but the situation immediately becomes very complicated after that point: there are 2 maximal matrices for $z_2(5)$, only 1 for $z_3(7)$, but 9 for $z_4(9)$.
	
	Some maximal matrices in the above table have been arranged to highlight a circulant (sub)matrix motif. This is most apparent for the $z_2(m)$ cases solved by Reiman's projective plane construction \cite{reiman}, since the resulting matrix can always be made circulant by a result of Singer \cite{singer}, but circulant matrices also appear elsewhere. For example, 23 is the smallest number $n$ above 21 for which five elements of $\mathbb Z/n\mathbb Z$ can be chosen so that their pairwise differences are all distinct -- $(0,1,3,8,14)$ is an example -- and taking cyclic shifts of any such set yields, as it turns out, the unique maximal matrix for $z_2(23)$. (Even after enforcing the constraints in section \ref{sec:method} the CNF instance still has exactly $6^6$ solutions, as counted by sharpSAT \cite{sharpSAT}; the matrix was verified to be unique by repeatedly shuffling and sorting its rows and columns, thereby reaching all $6^6$ solutions.)
	
	There is an asterisk in the row for $z_3(16)$ because even though its value has been shown to be exactly 128, the complete list of maximal matrices there has not yet been proven. The matrix shown is the only \textit{known} maximal one up to isomorphism.
	
	\section{Conclusion}
	
	The CNF instances we generated for each set of parameters did not split the problem into cases finer than specific combinations of row and column partitions. Combined with the use of just one processor at a time, this imposed a limit on how far our new results could reach with reasonable computational effort at around $z=100$. Parallelisation and further splitting outside the SAT solver (e.g.\ enumerating all possible ways to assign the first two rows and columns, each way leading to its own sub-case) as analysed in Heule \cite{schur5} would help, but those techniques would in turn allow more optimisations which we did not consider either -- argument D could exclude certain partial assignments without any further solving required, for example.
	
	Despite the limitations, our results represent a significant contribution towards Zarankiewicz's problem, both in the range of new values and in revealing the structure of maximal matrices -- previous work was mostly limited to the values, which by themselves follow no discernible pattern in general other than their strict monotonicity.
	
	\begingroup
	\sloppy
	\printbibliography
	\endgroup
\end{document}